\documentclass[titlepage]{amsart}
\usepackage{amsaddr}
\usepackage{amsthm,amsmath,amssymb,eucal}
\usepackage{geometry}
\usepackage{resizegather}
\usepackage{graphicx}
\usepackage{color}
\usepackage[normalem]{ulem}
\usepackage[latin1]{inputenc}
\usepackage{enumerate}

\newcommand{\ee}{\mathrm{e}}
\newcommand{\D}{\mathrm{d}}
\newcommand{\C}{\mathbb{C}}

\newcommand{\R}{\mathbb{R}}
\newcommand{\CC}{\mathcal{C}}
\newcommand{\DD}{\mathcal{D}}
\newcommand{\OO}{\mathcal{O}}
\newcommand{\RR}{\mathcal{R}}
\newcommand{\VV}{\mathcal{V}}

\newtheorem{claim}{Claim}[section]
\newtheorem{theorem}[claim]{Theorem}
\newtheorem{proposition}[claim]{Proposition}
\newtheorem{lemma}[claim]{Lemma}
\newtheorem{corollary}[claim]{Corollary}

\title[Soft quantum waveguides in three dimensions]{Soft quantum waveguides in three dimensions}

\author{Pavel Exner}

\address{{Doppler Institute for Mathematical Physics and Applied Mathematics, Czech Technical University,
B\v rehov{\'a} 7, 11519 Prague, Czechia} {\rm and} {Department of Theoretical Physics, Nuclear Physics Institute, Czech Academy of Sciences, 25068 \v{R}e\v{z} near Prague, Czechia}
}

\email{exner@ujf.cas.cz}

\date{\today}

\begin{document}

\begin{abstract}
We discuss a three-dimensional soft quantum waveguide, in other words, Schr\"odinger operator in $\R^3$ with an attractive potential supported by an infinite tube and keeping its transverse profile fixed. We show that if the tube is asymptotically straight, the distance between its ends is unbounded, and its twist satisfies the so-called Tang condition, the esential spectrum is not affected by smooth bends. Furthermore, we derive a sufficient condition, expressed in terms of the tube geometry, for the discrete spectrum of such an operator to be nonempty.
\end{abstract}

\maketitle

\section{Introduction} 
\setcounter{equation}{0}

Quantum mechanics of particles confined to regions of a tubular shape represents an interesting problem, both mathematically and from the practical point of view. On the one hand, there is an intriguing connection between spectral and transport properties of such systems and their geometry; for a survey we refer to \cite{EK15} and the bibliography therein. On the other hand, we have here a natural model to describe numerous physical systems investigated experimentally such as semiconductor nanowires, cold atom waveguides, and many others, see e.g. \cite{LCM99}.

Two frequently considered models of such a dynamics are hard-wall tubes, where the Hamiltonian is the Dirichlet Laplacian, and  `leaky wires', that is, singular Schr\"odinger operators formally written as $-\Delta-\alpha\delta(x-\Gamma)$, where $\Gamma$ is curve. As tools to describe real physical systems, both include idealizations. In the first case the tube boundary does now allow tunneling between different parts of such a guide. On the other hand, working with leaky wires we neglect the width of the guide. A natural way to overcome these deficiencies is to consider Schr\"odinger operators with a regular attractive potential in the form of a `ditch' of a fixed profile; a natural name \emph{soft quantum waveguides} for such systems was proposed recently \cite{Ex20}.

In the mentioned paper it was shown that two-dimensional soft waveguides share an important property with their more idealized counterparts, namely that a bend of the potential channel can give rise to bound states of such a Hamiltonian, although the existence condition derived in \cite{Ex20} lacks the universality known in the said two models. Other recent results concern the existence of a discrete spectrum in the example of a `bookcover' guide \cite{KKK21} and a spectral optimization for ring-shaped soft guides \cite{EL21}; a related two-body results can be found in \cite{EKP20}.

One has to add that `soft' constraints localizing particle motion to the vicinity of a more general Riemannian manifold have been considered before in different contexts; as a prime example one can mention the treatise \cite{WT14} and the literature there. In most of this work, however, the focus is rather on the behavior of such systems in the situation when such a constraint is `tightening', giving rise to an effective dynamics on the manifold; this approach roots in a decades-old physics idea how the quantization on manifolds can be dealt with \cite{KJ71, To88}. Particular cases of the zero-width of potential channels were also considered in the the theory of `leaky' quantum guides \cite{EI01, EK03}.

Returning to \cite{Ex20}, the discussion there provided a list of open questions. One of them concerned a three-dimensional counterpart of the soft waveguides studied there, which is the topic of the present short paper. The extension is nontrivial; the reason is that the geometry of the problem is more complicated in three dimensions. Except the case when the waveguide profile has a rotational symmetry in the normal planes, one has to impose a restriction on the guide twist in order to get the effective attractive interaction analogous to that in the two-dimensional situation, otherwise one also has a competing effect of an effective repulsion coming from the torsion. Such a general case is considerably more complicated and we will not address it here.

Our aim is to derive a condition which ensures the existence of a discrete spectrum in a potential channel which is bent, but asymptotically straight, and twisted, respecting the mentioned restriction. The tool we use is the Birman-Schwinger technique. Let us describe briefly the contents of the paper. In the next section we state the problem in proper terms and list the assumptions we make. After that, in Sec.~3, we will show that as long as the channel remains asymptotically straight and the distance between the tube ends is unbounded, the bending does not affect its essential spectrum. In Sec.~4 we first show that a bent channel has bound states if it is deep enough, then we recall the Birman-Schwinger principle and work it out for our purposes. The main result, a sufficient condition of a quantitative nature for the existence of a discrete spectrum, is stated and proved in Sec.~5. Finally, we add a few concluding remarks.

\section{The model} 
\setcounter{equation}{0}

Let $\Gamma$ be an infinite curve in $\R^3$ free of self-intersections, in other words, graph of a function $\Gamma:\:\R\to\R^3$; with an abuse of notation we employ the same symbol for both. We suppose that $\Gamma$ is parametrized by its arc length, $|\dot\Gamma(s)|=1$, and furthermore,
 \begin{enumerate}[(a)]
 \setlength{\itemsep}{0pt}
\item $\Gamma$ is $C^2$-smooth. \label{assa}
 \end{enumerate}
This allows us to define a global system of curvilinear coordinates in the vicinity of $\Gamma$. As long as $\ddot\Gamma(s)\ne 0$ the curve has the unique Frenet triad frame $(t,n,b)$. This is no longer true if $\ddot\Gamma$ vanishes, either at a point or because $\Gamma$ has a straight segment.
As it will become clear soon, however, our considerations involve rotations in the plane perpendicular to $\Gamma$ and the results do not change when the angle is shifted independently of $s$. Hence one can glue smoothly rotated local triads, and without loss of generality we may reason as if the Frenet frame existed globally. The three unit vectors satisfy, as functions of $s$, the Frenet formula
 $$ 
\left(\begin{array}{ccc} \dot t \\ \dot n \\ \dot b \end{array}\right)
= \begin{pmatrix} 0 & \gamma & 0 \\ -\gamma & 0 & \tau \\ 0 & -\tau & 0 \end{pmatrix}
\left(\begin{array}{ccc} t \\ n \\ b \end{array}\right),
 $$ 
where $\gamma,\,\tau$ are the curvature and torsion of $\Gamma$, respectively; recall that the knowledge of these functions allows us to reconstruct the curve uniquely up to Euclidean transformations. In addition, we suppose that the curve is such that
 \begin{enumerate}[(a)]
 \setcounter{enumi}{1}
 \setlength{\itemsep}{1pt}
\item $\gamma$ and $\tau$ are either of compact support, vanishing outside $[-s_0,s_0]$ for some $s_0>0$, or $\Gamma$ is $C^4$-smooth, $\tau,\,\dot\tau$ are bounded, and the curvature $\gamma$ together with its first and second derivatives tend to zero as $|s|\to\infty$, \label{assb}
\item $|\Gamma(s)-\Gamma(s')|\to \infty$ holds as $|s-s'|\to\infty$ (excluding, in particular, U-shaped curves). \label{assc}
 \end{enumerate}

The second ingredient is the support of the potential cross section. We suppose that it is an open precompact set $M\subset\R^2$ containing the origin of the coordinates and set $a:=\sup_{x\in M} |x|$. Without loss of generality we may suppose that $M$ is simply connected as the reader will see from assumption \eqref{assf} below. Given a function $\alpha:\R\to\R$, we define then the tubular region
 $$ 
 \Omega^\alpha_{\Gamma,M}:= x(\R\times M),\quad x(s,r,\theta):= \Gamma(s) - r\big[ n(s)\cos(\theta-\alpha) + b(s)\sin(\theta-\alpha) \big],
 $$ 
in particular, $\Omega_{\Gamma_0,M} := \R\times M$ corresponding to $x=\mathrm{id}$ is the tube built along a straight line. In the following we often drop the superscript and subscript if they are clear from the context. Furthermore, we assume that
 \begin{enumerate}[(a)]
 \setcounter{enumi}{3}
 \setlength{\itemsep}{0pt}
\item the map $x$ is injective, \label{assd}
\item $\dot\alpha=\tau$. \label{asse}
 \end{enumerate}
The first condition says, in other words, that the tube does not intersect itself. Locally the injectivity is guaranteed by the inequality $a\|\gamma\|_\infty<1$, however, the global injectivity requirement is stronger. Assumption~\eqref{assd} ensures the existence of a locally orthogonal system of coordinates in which points of $\Omega$ are uniquely parametrized by the arc length $s$, the distance $r$ from $\Gamma$, and the angle $\theta$.

Assumption~\eqref{asse}, on the other hand, usually referred to as \emph{Tang condition}, specifies a particular subset of such parametrizations in which the torsion does not couple the transverse and longitudinal coordinates, cf.~\eqref{tube metric t} below. It is a nontrivial restriction to the class of admissible tubes unless the cross section $M$ is a disc centered at the origin and, of course, the potential introduced below has the same symmetry. Spectral effects occurring in the absence of such a symmetry require a different technique and will be considered elsewhere.

As indicated in the introduction, we are interested in Schr\"odinger operators with an attractive potential supported in $\Omega_{\Gamma,M}$. Its profile is supposed to be
 \begin{enumerate}[(a)]
 \setcounter{enumi}{5}
 \setlength{\itemsep}{0pt}
 \item a nonzero $V\ge 0$ from $L^\infty(\R^2)$ with $\mathrm{supp}\,V\subset M$. \label{assf}
 \end{enumerate}
The support precompactness is vital for our reasoning, as well as the fact that the potential is nontrivial. The rest of the assumption, on the other hand, we make for simplicity of presentation. Most of our conclusions will remain valid without the positivity -- used only in Corollary~\ref{prop:smallwidth} -- and boundedness -- used in Proposition~\ref{prop:BS} -- as long as the operator \eqref{transop} below has a negative eigenvalue. Using the polar coordinates in the plane perpendicular to $\Gamma$, we set
 \begin{subequations}
 \begin{align}
 \tilde{V}:\: \Omega_{\Gamma,M} & \to\R_+,\quad \tilde{V}(x(s,r,\theta))=V(r,\theta), \label{potential} \\[.5em]
 H_{\Gamma,V} & = -\Delta-\tilde{V}(x)\,; \label{Hamiltonian}
 \end{align}
 \end{subequations}
in view of assumption \eqref{assf} the operator domain is $D(-\Delta)=H^2(\R^3)$. It is also useful to introduce the comparison operator on $L^2(\R^2)$,
 \begin{equation} \label{transop}
h_V = -\Delta-V(x)
 \end{equation}
with the domain $H^2(\R^2)$ which has by assumption \eqref{assf} a nonempty and finite discrete spectrum with
 \begin{equation} \label{infsph}
\epsilon_0 := \inf\sigma_\mathrm{disc}(h_V) = \inf\sigma(h_V)\in \big(-\|V\|_\infty,0\big),
 \end{equation}
where $\epsilon_0$ is a simple eigenvalue and the associated eigenfunction $\phi_0\in H^2(\R)$ can be chosen strictly positive; with an abuse of notation we will employ the same symbol for this function and its restriction to the set $M$.

We are interested in the spectrum of $H_{\Gamma,V}$. The relation \eqref{infsph} helps us to find it in the situation when the generating curve is a straight line because in that case the variables separate and we have
 \begin{equation} \label{straightsp}
\sigma(H_{\Gamma_0,V}) = \sigma_\mathrm{ess}(H_{\Gamma_0,V}) = [\epsilon_0,\infty).
 \end{equation}
We want to know what happens if $\Gamma$ satisfies assumptions~\eqref{assa}--\eqref{asse} being \emph{not straight} but at the same time \emph{asymptotically straight} in the sense of assumption~\eqref{assb}.

\section{The essential spectrum}\label{s:essential}
\setcounter{equation}{0}

Let us first show that under the hypotheses made about the geometry of the potential support the essential spectrum given by \eqref{straightsp} is preserved.

\begin{proposition} \label{prop:essential}
Under assumptions \eqref{assa}--\eqref{assf} we have $\sigma_\mathrm{ess}(H_{\Gamma,V}) = [\epsilon_0,\infty)$.
\end{proposition}
\begin{proof}
To begin with, we check that the lower bound of the essential spectrum cannot decrease, in other words, that
 \begin{equation} \label{infspess}
\inf\sigma_\mathrm{ess}(H_{\Gamma,V}) = \epsilon_0.
 \end{equation}
This can be done by Neumann bracketing; the argument is easier under the first part of assumption~\eqref{assb}. In that case we split $\R^3$ into four parts and impose Neumann condition at their boundaries. The two important ones are semiinfinite cylinders $\Sigma_\pm$ of radius $\rho>a$ the axes of which coincide with the straight parts of $\Gamma$, the third is a bounded region $\Sigma_\mathrm{c}$ complementing the two cylinders to a tubular neighborhood of $\Gamma$ containing $\Omega$, and the last one is the complement to the first three, $\Sigma_\mathrm{r}=\R^3 \setminus (\Sigma_\mathrm{c} \cup \Sigma_+ \cup\Sigma_-)$. The corresponding operator estimates $H_{\Gamma,V}$ from below and since it is a direct sum, $H^\mathrm{N}_{\Sigma_\mathrm{r}} \oplus H^\mathrm{N}_{\Sigma_\mathrm{c}} \oplus H^\mathrm{N}_{\Sigma_+} \oplus H^\mathrm{N}_{\Sigma_-}$, we may consider its parts separately. The first one is irrelevant because the operator is positive, the second one refers to a compact region and therefore it does not contribute to the essential spectrum. In the cylindrical regions the operators have separated variables being of the form $-(\partial^2_s)_\mathrm{N} \otimes h^\mathrm{N}_{V,\rho}$, where $h^\mathrm{N}_{V,\rho}$ is the restriction of $h_V$ to $L^2(\mathcal{B}_\rho(0))$ with Neumann boundary. By minimax principle we have $\inf\sigma_\mathrm{ess}(H_{\Gamma,V})\ge \epsilon(\rho)$, where the right-hand side is the lowest eigenvalue of this `transverse' operator. Mimicking the argument of \cite{DH93} one can check that $\epsilon(\rho)\to\epsilon_0$ holds as $\rho\to\infty$, in fact exponentially fast, and since in view of assumption~\eqref{assc} we can choose the splitting with $\rho$ arbitrarily large, we arrive at the relation \eqref{infspess}.

If $\Gamma$ is only asymptotically straight according to the second part of assumption \eqref{assb}, we can use Neumann bracketing again, but the use of separated variables requires a modification. We replace the above $\Sigma_\pm$ by curved cylinders written with the help of the curvilinear coordinates as $\Sigma_\pm:= x(I_\pm \times \mathcal{B}_\rho(0))$, where $I_\pm$ are halflines, $I_\pm:=\{\pm s:\, |s|>s_0\}$ for a suitable $s_0$. Regarding $\Sigma_\pm$ as Riemannian manifolds with a boundary, the corresponding metric tensor is
 \begin{equation} \label{tube metric t}
(g_{ij})= \left(\begin{array}{ccc}
\left(1+r\gamma\,\cos(\theta\!-\!\alpha)\right)^2
+r^2(\tau\!-\!\dot\alpha)^2\; &\;0\; & \;r^2(\tau\!-\!\dot\alpha)
\\ 0 & 1 & 0 \\ r^2(\tau\!-\!\dot\alpha) & 0 & r^2 \end{array}
\right),
 \end{equation}
cf.~\cite[Sec.~1.3]{EK15}. In particular, we have $g:=\det (g_{ij}) = r^2\left(1+r\gamma \,\cos(\theta\!-\!\alpha)\right)^2$ and the square root of this quantity is the Jacobian of the map $x$. Now assumption~\eqref{asse} comes to play: under it the metric tensor \eqref{tube metric t} becomes diagonal allowing us to separate variables at least asymptotically. To this aim, we use the unitary operator  $\tilde U:\:L^2(\Sigma_\pm)\to L^2(I_\pm \times \mathcal{B}_\rho(0), g^{1/2}\D s\,\D r\,\D\theta)$ associated with the map $x$, which turns $H_{\Gamma,V}\big|_{\Sigma_\pm}$ to the unitarily equivalent operator
 $$
\tilde H^{(\pm)}_{\Gamma,V}:= -g^{-1/2} \partial_s r^2 g^{-1/2} \partial_s - g^{-1/2} \partial_r
g^{1/2} \partial_r - g^{-1/2} \partial_\theta r^{-2} g^{1/2} \partial_\theta - V(r,\theta)
 $$
associated with the quadratic form
 \begin{equation} \label{tube form H}
\psi\mapsto \|(\tilde H^{(\pm)}_{\Gamma,V})^{1/2}\psi\|^2_g = \Vert r g^{-1/4}\partial_s \psi\|^2
+\|g^{1/4}\partial_r \psi\|^2 +\|g^{1/4} r^{-1}\partial_{\theta} \psi\|^2 - \|V^{1/2}\psi\|^2
 \end{equation}
defined on $H^1(I_\pm \times \mathcal{B}_\rho(0))$; the subscript $g$ denotes the norm in $L^2(I_\pm \times \mathcal{B}_\rho(0), g^{1/2}\D s\,\D r\,\D\theta)$. The first term on the right-hand side of \eqref{tube form H} is bounded from below by $(1+\rho\|\gamma_\pm\|_\infty)^{-1} \|r\partial_s \psi\|^2$, where $\gamma_\pm:= \gamma\big|_{I_\pm}$, the next two have a similar lower bound with the factor $1-\rho\|\gamma_\pm\|_\infty$ in the numerator. By assumption we have $\gamma(s)\to 0$ as $|s|\to\infty$. Combining this fact with assumption~\eqref{assc} we infer that one can choose the radius $\rho$ and the arc length cut-off $s_0$ both large and at the same time in such a way that $1\pm\rho\|\gamma_\pm\|_\infty$ is arbitrarily close to one; in this way we obtain relation \eqref{infspess} again.

Note that we have not used here the full strength of assumption~\eqref{assb}. The operator $\tilde H^{(\pm)}_{\Gamma,V}$ exists provided $\Gamma\in C^3$, and to analyze it through the quadratic form \eqref{tube form H} we need just $\Gamma\in C^2$ and no restrictions on the derivatives of $\gamma$ and $\tau$. The complete assumption makes it possible to pass to another pair of unitarily equivalent operators by removing the weight in the scalar product. This is achieved by using $U:\: L^2(\Sigma_\pm)\to L^2(I_\pm \times \mathcal{B}_\rho(0))$ defined as $U\psi= g^{1/4} \psi\circ x$. Explicitly, we have
 \begin{equation} \label{tube Ham}
H^{(\pm)}_{\Gamma,V} := U H_{\Gamma,V}\big|_{\Sigma_\pm} U^{-1}= -\partial_s h^{-2}\partial_s - h_V \big|_{L^2(\mathcal{B}_\rho(0))} + V_\Gamma(s,r.\theta)
 \end{equation}
where the effective potential is given by
 \begin{equation} \label{tube eff pot}
V_\Gamma(s,r,\theta)= -{\gamma^2\over 4 h^2} +{1\over 2}\,{h_{ss}\over
h^3} - {5\over 4}\,{h_s^2\over h^4}
 \end{equation}
with
 \begin{align*}
& h := g^{1/2}r^{-1} = 1+r\gamma\,\cos(\theta\!-\!\alpha),
\\ & h_s = r\gamma\tau\,\sin(\theta\!-\!\alpha) \,+\,
r\dot\gamma\, \cos(\theta\!-\!\alpha), \\ & h_{ss} =
r(\ddot\gamma-\gamma\tau^2)\,\cos(\theta\!-\!\alpha)
+r(2\dot\gamma\tau + \gamma\dot\tau)\,\sin(\theta\!-\!\alpha).
   \end{align*}
To make the description complete, one would have to specify the boundary conditions satisfied by functions from the domain of $H^{(\pm)}_{\Gamma,V}$ at the boundary of $\Sigma_\pm$ -- note that in contrast to Dirichlet conditions considered in~\cite[Sec.~1.3]{EK15} the Neumann ones are not preserved by the transformation $U$ -- but in the present context they do not matter, since we will work in the rest of the proof with functions the supports of which are separated from $\partial\Sigma_\pm$.

To finish the proof we have to check that there are no spectral gaps above $\epsilon_0$. To this aim we use Weyl's criterion; we pick functions $v\in C_0^\infty(\R)$ and $w\in C_0^\infty(\R^2)$ with the supports in $[-1,1]$ and $\mathcal{B}_1(0))$, respectively, such that their norms are $\|v\|=\|w\|=1$ and $v(s)=w(s)=1$ holds in the vicinity of zero, and put
 \begin{equation} \label{weylseq}
\psi(s,r,\theta) = \nu\sqrt{\mu}\: v(\mu(s-\tilde{s}_0))\,w(\nu r,\theta)\,\phi_0(r,\theta)\,\ee^{iks}
 \end{equation}
for $\mu,\,\nu>0$, a fixed $k\in\R$, and some $\tilde{s}_0>s_0$; the support of this function is contained in $\Sigma_+$ as long as $\nu<\rho$. Both cases of assumption~\eqref{assb} may be considered simultaneously; if $\Gamma$ is straight outside a compact, the used $s,r,\theta$ simply reduce to cylindrical coordinates in the vicinity of the straight parts. Since $\phi_0$ is the eigenfunction of $h_V$ corresponding to the eigenvalue $\epsilon_0$ we get
 \begin{align*} 
 (H^{(+)}_{\Gamma,V}-\epsilon_0-k^2)&\psi(s,r,\theta)= \nu\sqrt{\mu}\, \big[ \big(-\mu^2 v''(\mu(s-\tilde{s}_0)) - 2ik\mu v'(\mu(s-\tilde{s}_0))\big) w(\nu r,\theta)\phi_0(r,\theta) \\[.3em]
 & + v(\mu(s-\tilde{s}_0)) \big( -\nu^2 (\Delta w)(\nu r,\theta)\phi_0(r,\theta) - 2\nu (\nabla w)(\nu r,\theta)\cdot(\nabla \phi_0(r,\theta) \big) \big]\,\ee^{iks}
 \end{align*}
We denote the four terms on right-hand side as $f_j,\, j=1,2,3,4$, and estimate their norms. Using a simple change of variables we get
 $$ 
 \|f_1\|^2 = \int_{-1}^1 \D\xi \int_{\mathcal{B}_1(0)} \mu^4 v''(\xi)^2 w(\eta,\theta)^2 \phi_0(\nu\eta,\theta)^2\, \eta\,\D\eta\,\D\theta = \mu^4 \|v''\|^2 \phi_0(\underline{0})^2 \big(1+\OO(\nu) \big),
 $$ 
where $\underline{0}$ stands for the origin of coordinates, and consequently, $\|f_1\|=\OO(\mu^2)$. In a similar way we find the behavior of the other three terms, $\|f_2\|=\OO(\mu)$, $\|f_3\|=\OO(\nu^2)$, and $\|f_4\|=\OO(\nu)$, hence $\| (H^{(+)}_{\Gamma,V}-\epsilon_0-k^2)\psi\| \to 0$ holds as $\mu,\,\nu\to 0$, and the same is true if we replace $H^{(+)}_{\Gamma,V}$ by $H_{\Gamma,V}$ and $\psi$ by $U^{-1}\psi$. To prove that any number larger than $\epsilon_0$ belongs to $\sigma_\mathrm{ess}(H_{\Gamma,V})$, it is thus enough to find a family of regions threaded by the curve, mutually disjoint, and increasing without bounds, both in the radial and longitudinal directions, and supporting functions $\psi$ of the form \eqref{weylseq}. In view of assumptions~\eqref{assb} and \eqref{assc}, this is always possible.
\end{proof}

\section{Discrete spectrum: preliminaries}\label{s:prelim}
\setcounter{equation}{0}

As we have said, our main interest concerns the question whether a nontrivial geometry of the potential support can give rise to a discrete spectrum. Let us first note that in some situations one can get an affirmative answer almost for free. As an example, consider a soft \emph{flat-bottom} waveguide referring to the transverse potential
 \begin{equation} \label{flatbottom}
V \equiv V_\epsilon := \epsilon\chi_M,\quad \epsilon>0,
 \end{equation}
where $\chi_M$ is the indicator function of the set $M$. It is a common knowledge that a rectangular potential well yields Dirichlet boundary condition in the limit of infinite depth; recall that this claim can be given a mathematically rigorous meaning in the sense of norm-resolvent convergence, cf. \cite[Sec.~4.2.3]{DK05} or \cite[Sec.~21]{Si05}. This makes it possible to use a known result about the spectrum of Dirichlet Laplacian in bent tubes.

\begin{proposition} \label{prop:hardwall}
Suppose that $\Gamma$ is not straight and assumptions \eqref{assa}--\eqref{asse} are satisfied, then the discrete spectrum of $H_{\Gamma,V_\epsilon}$ is nonempty for all $\epsilon$ large enough.
\end{proposition}
\begin{proof}
The operator family $\{H_{\Gamma,V_\epsilon}:\,\epsilon\ge 0\}$ is clearly holomorphic of type (A) in the sense of Kato, and it is monotonous with respect to $\epsilon$. The same is true for operators $A_{\Gamma,V_\epsilon}:= H_{\Gamma,V_\epsilon}+\epsilon$ which are positive and form an increasing family, $A_{\Gamma,V_{\epsilon}} \ge A_{\Gamma,V_{\epsilon'}}$ for $\epsilon>\epsilon'$. Consequently, their eigenvalues $\lambda_j(\epsilon)$ are then continuously increasing functions of $\epsilon$, and as such they have limits as $\epsilon\to\infty$. The same applies to the threshold of their essential spectrum. In view of the norm-resolvent convergence mentioned above, the said limits are the respective spectral quantities of the (negative) Dirichlet Laplacian in the tube $\Omega$, the support of the potential $\tilde{V}_\epsilon$ with the profile~\eqref{flatbottom}. Under the hypotheses made, the assumptions of Theorem~1.3 in \cite{EK15} are satisfied, hence the limiting operator has a nonempty discrete spectrum below the continuum which starts at $\nu_1$, the ground state of $-\Delta^\mathrm{D}_M$. Consequently, for all sufficiently large $\epsilon$ we have $\sigma_\mathrm{disc}\big(A_{\Gamma,V_\epsilon}\big) \ne\emptyset$, and the same is true, of course, for the shifted operators $H_{\Gamma,V_\epsilon}$.
\end{proof}

Let us note that the flat-bottom assumption is, in fact, not needed; a similar result can be obtained for potentials if the type $V_\epsilon + W$ with $W\in L^\infty(M)$ by a modification of the reasoning that led to Theorem~1.3 in \cite{EK15}. We are not going to pursue this path, though, because we are looking for an existence result of a more quantitative nature.

The tool we are going to use to achieve this goal is the Birman-Schwinger principle: we rephrase the spectral problem of $H_{\Gamma,V}$ as analysis of the bounded operator
 \begin{equation} \label{BSop}
K_{\Gamma,V}(z) := \tilde{V}^{1/2} (-\Delta-z)^{-1} \tilde{V}^{1/2}
 \end{equation}
in $L^2(\R^3)$, where $\tilde{V}$ is the potential given by \eqref{potential}. Here $z\in\C\setminus\R_+$ is a fixed number from the resolvent set of the Laplacian; we are particularly interested in real values of the spectral parameter putting $z=-\kappa^2$ with $\kappa>0$. In that case the operator \eqref{BSop} is positive; assumption \eqref{asse} allows us to regard it as a map $L^2(\Omega) \to L^2(\Omega)$. The key result is the following:

\begin{proposition} \label{prop:BS}
$z\in\sigma_\mathrm{disc}(H_{\Gamma,V})$ holds if and only if $\,1\in\sigma_\mathrm{disc}(K_{\Gamma,V}(z))$. Moreover, the function $\kappa \mapsto K_{\Gamma,V}(-\kappa^2)$ is continuous, decreasing in $(0,\infty)$, and $\|K_{\Gamma,V}(-\kappa^2)\| \to 0$ holds as $\kappa\to\infty$.
\end{proposition}
\begin{proof}
The first claim is a particular case of a more general and commonly known result, see, e.g., \cite{BGRS97}. The continuity follows from the functional calculus and  we have
 $$ 
\frac{\D}{\D\kappa} (\psi,\tilde{V}^{1/2} (-\Delta+\kappa^2)^{-1}\, \tilde{V}^{1/2}\psi) = -2\kappa (\psi,\tilde{V}^{1/2} (-\Delta+\kappa^2)^{-2}\, \tilde{V}^{1/2}\psi) < 0
 $$ 
for any $\psi\in L^2(\Omega)$ provided $\tilde{V}^{1/2}\psi \ne 0$ which proves the monotonicity. Finally, the simple estimate $\|K_{\Gamma,V}(-\kappa^2)\| \le \kappa^{-2} \|V\|_\infty$ concludes the proof.
\end{proof}

It is also useful to recall that if $g$ is an eigenfunction of the operator \eqref{BSop} with eigenvalue one, the corresponding eigenfunction of the related $H_{\Gamma,V}$ is given by
 \begin{equation} \label{reconst}
 \phi(x) = \int_{\mathrm{supp}\,\tilde{V}} G_\kappa(x,x')\, \tilde{V}(x')^{1/2}g(x')\, \D x',
 \end{equation}
where $G_\kappa$ stands for the integral kernel of $(-\Delta+\kappa^2)^{-1}$. In our case we know the latter explicitly, which allows us to express the action of $K_{\Gamma,V}(z)$. In particular, for $z=-\kappa^2$ with $\kappa>0$ it is an integral operator in $L^2(\Omega)$ with the kernel
 $$ 
K_{\Gamma,V}(x,x';-\kappa^2) = \tilde{V}^{1/2}(x)\,\frac{\ee^{-\kappa|x-x'|}}{4\pi|x-x'|}\, \tilde{V}^{1/2}(x').
 $$ 

Now comes the moment when the part of assumption~\eqref{assf} requiring that the profile potential $V$ has a compact support, becomes important. It makes it possible to use the curvilinear coordinates introduced in the proof of Proposition~\ref{prop:essential} and translate the problem into analysis of an operator in the `straightened' tube. This is accomplished by means of the operator we have encountered already, however, this time we think of the straight tube as equipped with the cylindrical coordinates,
 \begin{equation} \label{straightening}
U: L^2(\Omega) \to L^2(\R\!\times\! M,\, r\,\D s\D r\D\theta), \;\; (U\psi)(s,r,\theta) = (1+r\gamma(s)\cos(\theta\!-\!\alpha(s)))^{1/2} \psi(x(s,r,\theta))\,;
 \end{equation}
we recall that by assumption~\eqref{asse} function $\alpha$ coincides with the derivative of the torsion; using it we pass from the Birman-Schwinger operator $K_{\Gamma,V}(-\kappa^2)$ to the unitarily equi\-valent one, $\RR^\kappa_{\Gamma,V} := U K_{\Gamma,V}(-\kappa^2) U^{-1}$, which is an integral operator on $L^2(\R\times M,\, r\,\D s\D r\D\theta)$ with the kernel
 \begin{equation} \label{uniteq}
\RR^\kappa_{\Gamma,V}(s,r,\theta;s',r',\theta') = W(s,r,\theta)^{1/2}\,\frac{\ee^{-\kappa|x-x'|}}{4\pi|x-x'|}\,W(s',r',\theta')^{1/2},
 \end{equation}
where $x=x(s,r,\theta)$, $x'=x(s',r',\theta')$, and the modified potential is given by
 \begin{equation} \label{modpot}
W(s,r,\theta) := \big(1+r\gamma(s)\cos(\theta-\dot\tau(s)\big)\,V(r,\theta).
 \end{equation}

Before proceeding to our main result, let us look  at the spectrum of the Birman-Schwinger operator in the straight guide situation, $\Gamma=\Gamma_0$. In that case its kernel equals
 \begin{equation} \label{straightker}
\RR^\kappa_{\Gamma_0,V}(s,\underline{x}) = V(s,\underline{x})^{1/2}\,\frac{\ee^{-\kappa|x-x'|}}{4\pi|x-x'|}\,V(s',\underline{x}')^{1/2},
 \end{equation}
where we write $x=(s,\underline{x})$ with the component $\underline{x}$ corresponding to the point with polar coordinates $(r,\theta)$ so that $|x-x'| = \big[ (s-s')^2 + (\underline{x}-\underline{x}')^2\big]^{1/2}$. With respect to the longitudinal variable $s$, the operator acts as a convolution, thus we can use the Fourier-Plancherel operator $F$ on $L^2(\R)$ to pass to a unitarily equivalent operator having the form of a direct integral,
 \begin{equation} \label{dirint}
 (F\otimes I) \RR^\kappa_{\Gamma_0,V} (F\otimes I)^{-1} = \int^\oplus_\R \RR^\kappa_{\Gamma_0,V}(p)\, \D p,
 \end{equation}
Using an integral identity for the Macdonald function, cf.~\cite[3.961.2]{GR07},
 $$ 
 \int_0^\infty \frac{\ee^{-\kappa\sqrt{\xi^2+\rho^2}}}{\sqrt{\xi^2+\rho^2}}\,\cos p\xi\,\D\xi = K_0\big(\rho\sqrt{\kappa^2+p^2}\big),
 $$ 
we find that the fibers in the decomposition \eqref{dirint} are integral operator on $L^2(M)$  with the kernels
\begin{equation} \label{1dBSkernel}
 \RR^\kappa_{\Gamma_0,V}(\underline{x},\underline{x}';p) = \frac{1}{2\pi}\, V(r,\theta)^{1/2}\, K_0\big(\sqrt{\kappa^2+p^2}|\underline{x}-\underline{x}'| \big)\, V(r',\theta')^{1/2}.
\end{equation}
This, however, is nothing else than the Birman-Schwinger kernel of the operator \eqref{transop} for the spectral parameter $z=-(\kappa^2+p^2)$. By assumption, $\epsilon_0$ is the principal eigenvalue of $h_V$, hence the decomposition \eqref{straightsp} in combination with Proposition~\ref{prop:BS} shows that the number $-\kappa^2=\epsilon_0+p^2$ belongs to the spectrum of $H_{\Gamma_0,V}$ for any $p\in\R$, in accordance with the relation \eqref{straightsp} obtained by separation of variables. At the same time, the operator with the kernel \eqref{straightker} satisfies
 $$ 
 \sup \sigma\big(\RR^{\kappa_0}_{\Gamma_0,V}\big) = 1,
 $$ 
where $\kappa_0:= \sqrt{-\epsilon_0}$. To see that assume that the left-hand side is larger than one, then by Proposition~\ref{prop:BS} there would exist a number $\tilde\kappa > \kappa_0$ such that $1\in \sigma(\RR^{\tilde\kappa}_{\Gamma_0,V})$, however, this would mean that $-\tilde\kappa^2 \in \sigma(H_{\Gamma_0,V})$ in contradiction to \eqref{straightsp}.

Finally, we note the relation between the eigenfunction $\phi_0$ of $h_V$ and the eigenfunction $g_0$ of $\RR^\kappa_{\Gamma_0,V}(0)$ corresponding to the unit eigenvalue. On the one hand, we have
 \begin{equation} \label{g_0}
 g_0(\underline{x}) = V^{1/2}(\underline{x})^{1/2}\phi_0(\underline{x}),
 \end{equation}
on the other hand, $\phi_0$ can be expressed in the way analogous to \eqref{reconst}; this allows us to write the generalized eigenfunction associated with the bottom of $\sigma(H_{\Gamma_0,V})$ as
 $$ 
 \psi_0(s,\underline{x}) = \phi_0(\underline{x}) = \frac{1}{2\pi} \int_M K_0\big(\kappa_0|\underline{x}-\underline{x}'|\big)\, V(\underline{x}')^{1/2}\,g_0(\underline{x}')\, \D\underline{x}'.
 $$ 

\section{The main result} 
\setcounter{equation}{0}

Now we are going to derive, with the help of Birman-Schwinger principle, a condition under which the soft quantum waveguide described by the Schr\"odinger operator $H_{\Gamma,V}$ possesses geometrically induced bound states.

\begin{theorem} \label{thm:boundstate}
We adopt assumptions \eqref{assa}--\eqref{assf} and set
 \begin{align*}
& \CC^\kappa_{\Gamma,V}(s,r,\theta;s',r',\theta') \\[-.3em] & \quad := \phi_0(r,\theta) V(r,\theta)\,\Big[\big(1+r\gamma(s)\cos(\theta\!-\!\dot\tau(s)\big)^{1/2}\, \frac{\ee^{-\kappa|x(s,r,\theta)-x(s',r',\theta')|}}{4\pi|x(s,r,\theta)-x(s',r',\theta')|}
\\[-.3em] & \qquad \times \big(1+r'\gamma(s')\cos(\theta'\!-\!\dot\tau(s')\big)^{1/2} \\[.3em] &
\quad -\, \frac{\ee^{-\kappa|x_0(s,r,\theta)-x_0(s',r',\theta')|}}{4\pi|x_0(s,r,\theta)-x_0(s',r',\theta')|}
\,\Big]\, V(r',\theta')\phi_0(r',\theta')
 \end{align*}
for all $(s,r,\theta),\,(s',r',\theta') \in \R\times M$, where $x,\,x_0$ are points of $\Omega_{\Gamma,M}$ and $\Omega_{\Gamma_0,M}$ with those coordinates, respectively. Then $\sigma_\mathrm{disc}(H_{\Gamma,V}) \ne \emptyset$ provided the following inequality holds for $\kappa = \kappa_0:= \sqrt{-\epsilon_0}$,
 \begin{equation} \label{BSsufficient}
 \int_{\R^2} \D s\,\D s' \int_{M\times M} \CC^{\kappa_0}_{\Gamma,V}(s,r,\theta;s',r',\theta')\, rr' \D r\,\D r'\,\D\theta\, \D\theta'> 0.
 \end{equation}
\end{theorem}
\begin{proof}
As in the two-dimensional case \cite{Ex20} we regard the bent shape of $\Omega$ as a perturbation of the straight guide; the transformation \eqref{straightening} allows us to do  that by comparing operators acting in the same Hilbert space. Proposition~\ref{prop:essential} shows that the perturbations preserves the essential spectrum, hence by Proposition~\ref{prop:BS} the operator $H_{\Gamma,V}$ has at least one eigenvalue below $\epsilon_0$ if and only if the Birman-Schwinger operator $\RR^{\kappa}_{\Gamma,V}$ has eigenvalue one for some $\kappa>\kappa_0$, and in view of the monotonicity of $\kappa\mapsto \RR^{\kappa}_{\Gamma,V}$, this is equivalent to $\sup\sigma(\RR^{\kappa_0}_{\Gamma,V})>1$. Thus it is sufficient to find a function $\psi_\eta\in L^2(\Omega^a_0)$ such that
 \begin{equation} \label{variation}
 (\psi, \RR^{\kappa_0}_{\Gamma,V}\psi) - \|\psi\|^2 > 0.
 \end{equation}
To find a suitable trial function we combine the generalized eigenfunction, associated with the edge of the continuum, with a mollifier which makes it an element of the Hilbert space. The latter makes a positive contribution to the energy; to assess it we consider first the straight case, $\Gamma=\Gamma_0$.
\begin{lemma} \label{l:straight}
Let $\psi_\eta\in L^2(\R\times M)$ be of the form $\psi_\eta(s,r,\theta) = h_\eta(s) g_0(r,\theta)$, where $h_\eta(s):=h(\eta s)$ with a function $h\in C_0^\infty(\R)$ such that $h(s)=1$ holds in the vicinity of $s=0$. Then
 $$ 
 (\psi_\eta,  \RR^{\kappa_0}_{\Gamma_0,V}\psi_\eta) - \|\psi_\eta\|^2 = \OO(\eta) \quad\text{as}\quad \eta\to 0+.
 $$ 
\end{lemma}
\begin{proof}
Since $g_0$ is by \eqref{dirint} and \eqref{1dBSkernel} an eigenfuction of $\RR^{\kappa_0}_{\Gamma_0,V}(0)$, the second term on the left-hand side can written as $-\|h_\eta\|^2 (g_0, \RR^{\kappa_0}_{\Gamma_0,V}(0) g_0)$. Combining the two indicated relations with \eqref{straightker} we have thus to estimate the expression
 \begin{align*}
 \frac{1}{2\pi} \int_{M\times M} & g_0(r,\theta)V(r,\theta)^{1/2} \bigg[\int_\R |\hat h_\eta(p)|^2\, K_0\big(\sqrt{\kappa_0^2+p^2}|\underline{x}-\underline{x}'| \big)\,\D p
 - \|h_\eta\|^2 K_0\big(\kappa_0|\underline{x}-\underline{x}'| \big) \bigg] \\[.3em] & \times V(r',\theta')^{1/2} g_0(r',\theta')\,rr'\,\D r\D r'\D\theta \D\theta',
 \end{align*}
where $\underline{x},\,\underline{x}'$ are again the points of $M$ with the coordinates $(r,\theta)$ and $(r',\theta')$, respectively. The potential $V$ is bounded by assumption and the same is true for the function \eqref{g_0}, so it is sufficient to check that
 $$ 
 \int_\R |\hat h_\eta(p)|^2\, K_0\big(\sqrt{\kappa_0^2+p^2}|\underline{x}-\underline{x}'| \big)\,\D p
 - \|h_\eta\|^2 K_0\big(\kappa_0|\underline{x}-\underline{x}'| \big) = \OO(\eta)
 $$ 
holds as $\eta\to 0$. The Fourier transform of $h_\eta$ is $\hat h_\eta(p) = \frac1\eta\, \hat h\big( \frac{p}{\eta}\big)$ which, in combination with the mean value theorem, allows us to write the first term on the left-hand side as
 $$ 
 \frac1\eta\, \int_\R |\hat h(\zeta)|^2\, K_0\big(\sqrt{\kappa_0^2+\eta^2\zeta^2}|\underline{x}-\underline{x}'|\big)\,\D\zeta
 = \frac1\eta\,  \big( K_0\big(\kappa_0|\underline{x}-\underline{x}'| \big) + \OO(\eta^2) \big)\|\hat h\|^2,
 $$ 
and using further the relations $\|\hat h\|=\|h\|$ and $\|h_\eta\|^2 = \frac1\eta\,\|h\|^2$, we conclude the proof.
\end{proof}

With this preliminary we can prove Theorem~\ref{thm:boundstate}. Consider the difference of the Birman-Schwinger operators referring to the bent and straight case,
 \begin{equation} \label{difference}
 \DD^\kappa_{\Gamma,V} := \RR^\kappa_{\Gamma,V} - \RR^\kappa_{\Gamma_0,V}
 \end{equation}
which is by \eqref{uniteq} and \eqref{straightker} an integral operator on $L^2(\R\!\times\! M,\, r\,\D s\D r\D\theta)$ with the kernel
 $$ 
 \DD^\kappa_{\Gamma,V}(s,r,\theta;s',r',\theta') = W(s,u)^{1/2}\,\frac{\ee^{-\kappa|x-x'|}}{4\pi|x-x'|}\, W(s',u')^{1/2} - V(u)^{1/2} \,\frac{\ee^{-\kappa|x_0-x_0'|}}{4\pi|x_0-x_0'|}\, V(u')^{1/2}
 $$ 
Using \eqref{variation} and \eqref{difference} in combination with Lemmma~\ref{l:straight} we infer that $\sup \sigma(\RR^{\kappa_0}_{\Gamma_0,V}) > 1$ would hold provided
 $$ 
 \lim_{\eta\to 0} (\psi_\eta, \DD^{\kappa_0}_{\Gamma,V} \psi_\eta) > 0,
 $$ 
and given our choice of $h_\eta$, this happens in view of the dominant convergence theorem if
 $$ 
 \int_{\R^2} \D s\D s' \int_{M\times M} g_0(r,\theta)\, \DD^{\kappa_0}_{\Gamma,V}(s,r,\theta;s',r',\theta')\, g_0(r'\theta')\,rr'\,\D r\D r'\D\theta \D\theta' > 0,
 $$ 
however, in view of \eqref{modpot} and \eqref{g_0}, this is nothing but the inequality \eqref{BSsufficient}.
\end{proof}

The sufficient condition \eqref{BSsufficient} is of a quantitative nature. Its application requires to evaluate the integral in the condition~\eqref{BSsufficient} which may not be an easy task in particular cases. Even without it, however, we can make conclusions about the existence of a discrete spectrum, for instance:

\begin{corollary} \label{prop:smallwidth}
Let $\VV_{\epsilon_0}$ be a family of profile potentials $V$ satisfying assumption \eqref{assf} and such that $\inf\sigma(h_V)=\epsilon_0$. Then to any $\epsilon_0>0$ there exists an $a_0(\epsilon_0)$ such that $\sigma_\mathrm{disc}(H_{\Gamma,V}) \ne \emptyset$ holds for all $V\in\VV_{\epsilon_0}$ with $a:=\sup_{x\in M} |x| <a_0(\epsilon_0)$.
\end{corollary}
\begin{proof}
It is straightforward to check that under our assumptions the integration in \eqref{BSsufficient} can be performed in any order, hence the expression in question can be rewritten as
 \begin{equation} \label{reformulated}
 \frac{1}{2\pi}\,\int_{M\times M} \phi_0(r,\theta)V(r,\theta)\,F(r,\theta;r',\theta')\, V(r',\theta')\phi_0(r',\theta')\,rr'\,\D r\D r'\D\theta \D\theta',
 \end{equation}
where
 \begin{align*}
F(r,\theta;r',\theta') := & \int_{\R^2} \big[(1+u\gamma(s))^{1/2}\,K_0(\kappa_0|x(s,r,\theta)-x(s',r',\theta')|)\,(1+u'\gamma(s'))^{1/2} \\[.3em] & \qquad -  K_0(\kappa_0|x_0(s,u)- x_0(s',u')|)\,\big]\, \D s\D s'.
 \end{align*}
The function $F(\cdot,\cdot)$ is well defined as long as assumption \eqref{asse} guaranteeing the existence of the curvilinear coordinates is satisfied and it is continuous. Furthermore, we have
 \begin{equation} \label{onthecurve}
F(\underline{0},\underline{0}) := \int_{\R^2} \big[K_0(\kappa_0|\Gamma(s)-\Gamma(s')|) -  K_0(\kappa_0|s-s'|)\big]\, \D s\,\D s' > 0,
 \end{equation}
where $\underline{0}$ stands again for the origin of the coordinates in $M$. Recall that the curve $\Gamma$ is parametrized by its arc length so that $|\Gamma(s)-\Gamma(s')| \le |s-s'|$, and since $\Gamma$ is not straight by assumption, there is an open set on which the inequality is sharp; the inequality \eqref{onthecurve} follows because the function $K_0(\cdot)$ is decreasing in $(0,\infty)$. The continuity then implies the existence of a neighborhood $\mathcal{U} \times \mathcal{U}$ of the point $(\underline{0},\underline{0})$ on which $F(r,\theta;r',\theta')$ is positive, and that in view of \eqref{reformulated} in combination with the positivity of $\phi_0 V$ yields the claim.
\end{proof}

\section{Concluding remarks} 
\setcounter{equation}{0}

The present result leaves various questions open. To begin with, the obtained sufficient condition is substantially weaker than the corresponding results for hard-wall tubes and leaky wires, in particular, it is not clear from it whether bound states do generally exist in shallow potential ditches. Of course, the Birman-Schwinger technique we employed here is not the only tool one could use. A natural alternative would be to seek variational estimates to the original Schr\"odinger operator $H_{\Gamma,V}$. However, this is not easy either, the only example available so far concerns a rather particular case \cite{KKK21} and its extension to more general situations represents a challenge.

Furthermore, in contrast to the two-dimensional situation considered in \cite{Ex20} we have here an important geometric restriction expressed by assumption~\eqref{asse} which allowed us to separate effectively the longitudinal and transverse motion. It is not just a technical obstacle: in analogy with the properties of the Dirichlet Laplacian in twisted tubular regions \cite[Sec.~1.7]{EK15} one expects that a twist of the potential support violating the Tang condition will give rise to a repulsive effective interaction which, in particular, might lead to a spectral stability analogous to \cite{EKK08}.

Another possible extension of the current analysis concerns spectral properties of Schr\"odinger operator with attractive potential of a fixed transverse profile supported in the vicinity of an infinite surface which is curved but asymptotically planar. One expects that the nontrivial geometry can again produce bound states, and in addition, that in distinction to tubes the global geometry of the interaction support may play role. In fact, we know that such a curvature-induced discrete spectrum may be infinite, as is the case in an example in which the generating surface is conical \cite{EKP20}, however, a number of questions remains open here.

\subsection*{Data availability statement}

Data are available in the article.

\subsection*{Conflict of interest}

The author has no conflict of interest.

\subsection*{Acknowledgements}

Thanks go to the anonymous referee for useful comments. The research was supported by the Czech Science Foundation within the project 21-07129S and by the EU project CZ.02.1.01/0.0/0.0/16\textunderscore 019/0000778.


\begin{thebibliography}{99}
\setlength{\itemsep}{2pt}


\bibitem{BGRS97}
Bulla, W., Gesztesy, F., Renger, W., Simon, B., ``Weakly coupled bound states in quantum waveguides'', \emph{Proc. Amer. Math. Soc.} \textbf{127}, 1487--1495 (1997).

\bibitem{DH93}
Dauge, M., Helffer, B., ``Eigenvalue variation. I. Neumann problem for Sturm-Liouville operators'', \emph{J. Diff. Eqs} \textbf{104}, 243--262  (1993).

\bibitem{DK05}
Demuth, M., Krishna, M., \emph{``Determining Spectra in Quantum Theory''}, Birkh\"auser, Boston 2005.

\bibitem{EKP20}
Egger, S., Kerner, J., Pankrashkin, K.,  ``Discrete spectrum of Schr\"odinger operators with potentials concentrated near conical surfaces'', \emph{Lett. Math. Phys.} \textbf{110}, 945--968  (2020).

\bibitem{EKK08}
Ekholm, T., Kova\v r\'{\i}k, H., Krej\v ci\v{r}\'{\i}k, D., ``A Hardy inequality in twisted waveguides'', \emph{Arch. Rat. Mech. Anal.} \textbf{188}, 245--264 (2008).

\bibitem{Ex20}
Exner, P., ``Spectral properties of soft quantum waveguides'', \emph{J. Phys. A: Math. Theor.} \textbf{53}, 355302  (2020); corrigendum  {\bf 54}, 099501 (2021).

\bibitem{EI01}
Exner, P., Ichinose, T., ``Geometrically induced spectrum in curved leaky wires'', \emph{J. Phys. A: Math. Gen.} \textbf{34}, 1439--1450  (2001).

\bibitem{EK03}
Exner, P., Kondej, S., ``Bound states due to a strong $\delta$ interaction supported by a curved surface'', \emph{J. Phys. A: Math. Gen.} \textbf{36}, 443--457 (2003).

\bibitem{EK15}
Exner, P., Kova\v{r}\'{\i}k, H., \emph{``Quantum Waveguides''}, Springer International, Cham 2015.

\bibitem{EL21}
Exner, P., Lotoreichik, V., ``Optimization of the lowest eigenvalue of a soft quantum ring'', {\em Lett. Math. Phys.} {\bf 111}, 28 (2021).

\bibitem{GR07}
Gradshtein, I.S., Ryzhik, I.M., \emph{``Tables of Integral, Series, and Products''}, 7th edition, Academic Press, New York 2007.

\bibitem{KKK21}
Kondej, S., Krej\v{c}i\v{r}\'{\i}k, D., K\v{r}\'{\i}\v{z}, J., ``Soft quantum waveguides with a explicit cut locus'', \emph{J. Phys. A: Math. Theor.} \textbf{54}, 30LT01  (2021).

\bibitem{KJ71}
Koppe, H., Jensen, H., ``Das Prinzip von d'Alembert in der klassischen Mechanik und in der Quantentheorie'', in \emph{Sitzungber. der
Heidelberger Akad. der Wiss., Math.-Naturwiss. Klasse} \textbf{5}, 127--140  (1971).

\bibitem{LCM99}
Londergan, J.T., Carini, J.P., Murdock, D.P., \emph{``Binding and Scattering in Two-Dimensional Systems. Applications to Quantum
Wires, Waveguides and Photonic Crystals''}, Springer LNP~m60, Berlin 1999.

\bibitem{Si05}
Simon, B., \emph{``Functional Integration in Quantum Physics''}, 2nd edition, AMS Chelsea, Providence, R.I. 2005.

\bibitem{To88}
Tolar, J., ``On a quantum mechanical d'Alembert principle'', in \emph{Group Theoretical Methods in Physics}, Lecture Notes in Physics, vol.~313, Springer, Berlin 1988; pp.~268--274.

\bibitem{WT14}
Wachsmuth, J., Teufel, S.: \emph{``Effective Hamiltonians for Constrained Quantum Systems''}, Mem. AMS, vol.~230, Providence, R.I. 2014.

\end{thebibliography}
\end{document}